\documentclass[10pt]{amsart}

\usepackage[utf8]{inputenc}
\usepackage[T2A]{fontenc}
\usepackage[russian, main=english]{babel}

\usepackage{amsmath,amsthm,amsfonts, bbm, amssymb, hyperref,graphicx,color}
\usepackage{caption, subcaption}
\def\[#1\]{\begin{align}#1\end{align}}
\def\(#1\){\begin{align*}#1\end{align*}}

\newcommand{\ii}{\mathbbm{i}}
\newcommand{\jj}{\mathbbm{j}}
\newcommand{\kk}{\mathbbm{k}}

\renewcommand{\C}{\mathbb{C}}

\newcommand{\R}{\mathbb{R}}
\newcommand{\Z}{\mathbb{Z}}

\newcommand{\id}{\operatorname{id}}
\newcommand{\X}{\mathbb X}

\newcommand{\norm}[1]{\left\vert #1 \right \vert}	
\newcommand{\Norm}[1]{\left\Vert #1 \right \Vert}

\renewcommand{\Re}{\operatorname{Re}}
\renewcommand{\Im}{\operatorname{Im}}

\newcommand{\Isom}{\operatorname{Isom}}

\newcommand{\comment}[1]{}

\newcommand{\ignore}[1]{}
\newcommand{\twovector}[2]{
\left[\begin{array}{c}
#1\\#2	
\end{array}\right]}
\newcommand{\threevector}[3]{
\left[\begin{array}{ccc}
#1\\#2\\#3	
\end{array}\right]}

\newtheorem{thm}{Theorem}

\newtheorem{lemma}[thm]{Lemma}
\newtheorem{cor}[thm]{Corollary}

\numberwithin{thm}{section}

\theoremstyle{definition}
\newtheorem{defi}[thm]{Definition}

\theoremstyle{definition}
\newtheorem{example}[thm]{Example}
\newtheorem{remark}[thm]{Remark}

\numberwithin{equation}{section}

\newcommand{\st}{\;:\;}
\newcommand{\mst}{\;:\;}

\newcommand{\Mod}{\mathcal M}

\newcommand{\Zee}{\mathcal Z}

\newcommand{\Sph}{\mathbb S}

\newcommand{\twomatrix}[4]{\begin{bmatrix}
#1 & #2 \\
#3 & #4 
\end{bmatrix}}
\newcommand{\inversiontwomatrix}{\twomatrix 0 1 1 0}
\newcommand{\translationtwomatrix}[1]{\twomatrix 1 {#1} 0 1}

\title{Convergence of improper Iwasawa Continued Fractions}
\author[A. Lukyanenko]{Anton Lukyanenko}
\address{
Department of Mathematics\\
George Mason University\\
4400 University Drive, MS: 3F2\\
Fairfax, Virginia 22030}
\email{anton@lukyanenko.net}
\author[J. Vandehey]{Joseph Vandehey}
\address{
Department of Mathematics\\
University of Texas at Tyler\\
Tyler, TX 75799
}
\email{jvandehey@uttyler.edu}

\keywords{Convergence, continued fractions, complex continued fractions, Iwasawa continued fractions, Heisenberg group, octonions, quaternions.\\
\indent \emph{2020 Subject Classification:} Primary 11K50, Secondary 11R52, 53C17.
}

\begin{document}

\begin{abstract}
    We prove the convergence of a wide class of continued fractions, including generalized continued fractions over quaternions and octonions. Fractional points in these systems are not bounded away from the unit sphere, so that the iteration map is not uniformly expanding. We bypass this problem by analyzing digit sequences for points that converge to the unit sphere under iteration, expanding on previous methods of Dani-Nogueira.
\end{abstract}

\maketitle

\section{Introduction}

The classical continued fraction (CF) expansion of an irrational real number $x\in[0,1)$ is given by the expression
\[
x= \cfrac{1}{a_1+\cfrac{1}{a_2+\cfrac{1}{a_3+\cfrac{1}{\ddots}}}},\label{eq:basicexpansion}
\]
for some sequence of positive integers $a_i$. Variant CF expansions may introduce minus signs, allow the numerators to vary as well as the denominators, or---most relevant here---expand the setting to other underlying spaces such as the complex numbers, with a corresponding change to the set of allowable digits $a_i$ . Before further questions can be considered, one must first ask whether such a CF expansion converges, and this can be a surprisingly deep question in its own right. The question of when \eqref{eq:basicexpansion} converges for sequences $a_i\in \mathbb{Z}$ (not just $a_i\in\mathbb{N}$) was only recently solved by Short and Stanier \cite{SS}.

For the analytic theory of continued fractions, where both numerators and denominators are allowed to be arbitrary complex numbers (not even belonging to a ring of integers), famous convergence results include the \'Sleszy\'nski-Pringsheim Theorem \cite{sleshinskiy}, Van Vleck \cite{VV}, Worpitzky \cite{Worpitzky}, and many other subsequent papers, including \cite{BS,BB,JT,SW}. Wall \cite{WallBook} and Lorentzen and Waadeland \cite{LW} devote a significant portion of their books to this problem. Convergence results have also been obtained for continued fractions whose digits belong to far more abstract spaces, such as operators on a Banach space \cite{PH} or matrices \cite{ZZ}.

In this paper, we will focus on continued fraction expansions whose digits come from a ring of integers or similarly nice lattice. Here, Tietze \cite{Tietze} famously proved the convergence of all \emph{semi-regular continued fractions}, a large class of CF expansions with real integer digits. Rosen \cite{Rosen} studied continued fractions related to the Hecke groups. Dani \cite{Dani2015} proved the convergence of a wide variety of complex continued fractions, including those that use Gaussian integers and Eisenstein integers. In Schweiger's book \cite{SchweigerBook} convergence is discussed for a large class of multi-dimensional continued fraction algorithms.

In many of these cases, the convergence of the representation is ultimately based on the fact that the inversion map $\iota:x\mapsto 1/x$ is an expanding map inside the unit ball. 

Let $T_{a_1}x=\frac{1}{x}-a_1$ be the map that removes the first digit from the CF expansion \eqref{eq:basicexpansion} of $x$, and consider the orbit $\{T_{a_n}\cdots T_{a_1}x\}$. If $\norm{T_{a_n}\cdots T_{a_1} x}<1-\epsilon$ for some $\epsilon>0$ and for all $n\ge 0$, then $\iota$ will always be \emph{uniformly} expanding along this orbit, and often very simple arguments will suffice to prove the convergence of the expansion \eqref{eq:basicexpansion}.\footnote{We will ignore those countably many points $x$ for which we ever see $T_{a_n}\cdots T_{a_1} x=0$.} CF algorithms which obey this norm bound were called \emph{proper} by the authors \cite{lukyanenko_vandehey_2022}, who showed that proper algorithms on the quaternions, octonions, and other Iwasawa inversion spaces would converge. Here, the only difficulty is in showing that division by zero does not occur.


Recently, Dani-Nogueira \cite[Theorem~3.7]{DN2014} weakened this condition to require only that $\norm{T_{a_n}\cdots T_{a_1} x}<1$ for all $n\ge 0$, and under this assumption proved the convergence of any  CF representations of a point with Gaussian integer digits. In \cite{Dani2016}, Dani further extends this to consider digits belonging to any discrete subring of $\C$ containing $1$. As a consequence of our Main Theorem \ref{thm:IwasawaMain}, we are able to extend their result to \emph{improper} CF algorithms over quaternions and octonions (for concrete examples of such algorithms, see Section \ref{sec:examples}). 

Before stating our main result, we recall some standard notation:
\begin{defi}Let $k$ be a finite-dimensional \emph{real division algebra}, i.e.,~the reals, complex numbers, quaternions, or octonions. We think of $k$ as $\R^d$ with a basis $1=e_0, e_1, \ldots, e_d$ for $d=1,2,4,8$, respectively, with the standard multiplicative structure. In particular, we write $x=a_0 + \sum_{i=1}^{d-1} a_i e_i$, writing $e_1=\ii, e_2=\jj, e_3=\kk$ when working with complex or quaternionic numbers. One then has the \emph{real part of $x$}, $\Re(x)=a_0$, the \emph{imaginary part of x}, $\Im(x)=x-\Re(x)$ and \emph{conjugate of x}, $\overline x = \Re(x)-\Im(x)$. For $p, q\in k$ we write $\frac{p}{q}=pq^{-1}$, where $q^{-1}$ is the multiplicative inverse of $q$. \linebreak
The Euclidean norm on $\R^d$ can be expressed as $\norm{x}^2 = x \overline x=\sum_{i=0}^7 a_i^2$, providing the Euclidean distance function $d(\cdot, \cdot)$ and corresponding notion of convergence. In particular, one writes $x=\lim_{i\rightarrow \infty} x_i$ if $\lim_{i\rightarrow \infty} d(x_i, x)=\norm{x_i-x}=0$.\linebreak
A subring $\Zee$ of $k$ is \emph{discrete} if it is a discrete subset of $\R^d$ with the Euclidean topology.  Given a digit $a\in \Zee$, we write $T_ax=\frac{1}{x}$ and $T_a^{-1}x = \frac{1}{x+a}$.
\end{defi}


\begin{thm}\label{thm:DAmain}
Let $k$ be a real division algebra and $\Zee\subset k$ a discrete ring closed under conjugation. Given a sequence $a_n\in \Zee$ and $x_0\in k$, suppose that $\norm{T_{a_n}\cdots T_{a_1}x_0}<1$ for all $n\geq 0$. Then the associated continued fraction converges to $x_0$:
\(x_0=\lim_{n\rightarrow \infty} T^{-1}_{a_1}\cdots T^{-1}_{a_n}0=\lim_{n\rightarrow \infty} \cfrac{1}{a_1 + \cfrac{1}{\ddots+\cfrac{1}{a_n}}}.\)
\end{thm}

While it is common in continued fraction theory to study a particular CF algorithm that would assign to any given $x_0$ a unique sequence $a_n\in \Zee$, Theorem \ref{thm:DAmain} makes no such assumption. A single $x_0\in k$ could have many associated sequences\footnote{Indeed, the typical $x_0$ often has uncountably many expansions.}, and provided each of them satisfies the necessary condition, the CF expansions will all converge.



In full generality, our main result (Theorem \ref{thm:IwasawaMain}) will allow us to obtain convergent CF expansions on any Iwasawa inversion space (e.g., $\R^n$ or the Heisenberg groups), allow more general inversions (e.g., $\iota:x\mapsto -1/x$), and digit choices (as with folded CFs). In particular, we prove for the first time the convergence of CF \emph{algorithms} in the (7-dimensional) Iwasawa inversion spaces
$\X^1_\mathbb{H}$; see Section \ref{sec:examples} for details.

Furthermore, it follows from the proof of Theorem \ref{thm:DAmain} that, under additional assumptions, we may weaken the condition $\norm{T_{a_n}\cdots T_{a_1}x_0}<1$ in Theorem \ref{thm:DAmain} (with an analogous result extending Theorem \ref{thm:IwasawaMain}):
\begin{thm}\label{thm:stronglyirrational}
Let $k$ be a real division algebra and $\Zee\subset k$ a discrete ring closed under conjugation. Given a sequence $a_n\in \Zee$ and $x_0\in k$, suppose that:
\begin{enumerate}
    \item For each $n\geq 0$ and $a\in \Zee \setminus\{0\}$, $\prod_{i=0}^n \norm{T_{a_i}\cdots T_{a_1}x_0}\neq \norm{a}$
    \\(e.g., $x_0$ is transcendental over $\Zee$), and
    \item $\prod_{i=0}^\infty \norm{T_{a_n}\cdots T_{a_1}x_0} = 0$.
\end{enumerate}
Then the associated continued fraction converges to $x$.
\end{thm}

\subsection{An outline of the method}
\label{sec:outline}

Theorem \ref{thm:DAmain} (and Theorem \ref{thm:IwasawaMain}, after appropriate changes in notation) follow immediately from three lemmas, all of which have the same assumptions as the theorem. We start by stating the lemmas, and then prove the first lemma and sketch the proofs of the second and third lemmas, which will make up the rest of this paper.
\begin{lemma}
\label{lemma:prodIntro}
The distance from the CF approximates $T^{-1}_{a_1}\cdots T^{-1}_{a_n}0$ to $x_0$ satisfies
\[
\label{eq:prodIntro}
\norm{T^{-1}_{a_1}\cdots T^{-1}_{a_n}0-x_0}= \prod_{i=0}^n \norm{T_{a_i}\cdots T_{a_1}x_0} \prod_{i=0}^{n-1} \norm{T^{-1}_{a_{i+1}}\cdots T^{-1}_{a_n}0},
\]
assuming the terms of the latter product are well-defined.
\end{lemma}

\begin{lemma}
\label{lemma:nonzero q}
The product \[\label{eq:product definition of qn}
\prod_{i=0}^{n-1} \norm{T^{-1}_{a_{i+1}}\cdots T^{-1}_{a_n}0}
\]
is well-defined and uniformly bounded above over all $n$.
\end{lemma}

\begin{lemma}
\label{thm:non-convergence}
The sequence $\norm{T_{a_n}\cdots T_{a_1}x_0}$ does not converge to 1.
\end{lemma}

We now comment on the three lemmas, starting with the proof of Lemma \ref{lemma:prodIntro}:
\begin{proof}[Proof of Lemma \ref{lemma:prodIntro}]
The identity $x-y=x\left(\frac{1}{y}-\frac{1}{x}\right)y$ holds over any real division algebra (note that the octonions are alternative, i.e., associative on subalgebras generated by two elements), giving
\begin{equation}
\label{eq:inversionidentity}
    \norm{x-y}=\norm{1/x-1/y}\norm{x}\norm{y}.
\end{equation}
One calculates:
\begin{align*}
     \norm{T^{-1}_{a_1}\cdots T^{-1}_{a_n}0 - x_0}&=\norm{\frac{1}{T^{-1}_{a_1}\cdots T^{-1}_{a_n}0} - \frac{1}{x_0}}\norm{T^{-1}_{a_1}\cdots T^{-1}_{a_n}0}\norm{x_0}\\&=
     \norm{\left(\frac{1}{T^{-1}_{a_1}\cdots T^{-1}_{a_n}0}-a_1\right) - \left(\frac{1}{x_0}-a_1\right)}\norm{T^{-1}_{a_1}\cdots T^{-1}_{a_n}0}\norm{x_0}\\&=
     \norm{T^{-1}_{a_2}\cdots T^{-1}_{a_n}0-T_{a_1}x_0}\norm{T^{-1}_{a_1}\cdots T^{-1}_{a_n}0}\norm{x_0}.
\end{align*}
Iterating $n$ times gives \eqref{eq:prodIntro}. Note that the inversion identity \eqref{eq:inversionidentity} continues to hold in the more general Iwasawa setting, see Lemma \ref{lemma:inverseFormulaSiegel}.
\end{proof}

To explain Lemma \ref{lemma:nonzero q} better, if one rewrites the finite CF expansion $T^{-1}_{a_1}\dots T^{-1}_{a_n}0$ in a specific way as a rational number $\frac{p_n}{q_n}$, then it can be shown that the product \eqref{eq:product definition of qn} is equal to $|q_n|^{-1}$. Thus Lemma \ref{lemma:nonzero q} is equivalent to stating that $q_n$ is bounded away from zero. While these equivalencies are well-known in the real and complex cases, more care is required in other cases.

We previously proved Lemma \ref{lemma:nonzero q} for the Heisenberg group $\X_\C^1$ in \cite[Lemma~3.20]{LV2015}, as a culmination of several complicated matrix identities. We then proved Lemma \ref{lemma:nonzero q} for all Iwasawa spaces in \cite[Lemma~4.4]{lukyanenko_vandehey_2022}, where we employed the structure theory of hyperbolic manifolds. In this paper we will provide a substantially simplified proof. 

For readers interested in individual cases, we furthermore prove Lemma \ref{lemma:nonzero q} separately in three different settings, with increasing generality. In Section \ref{sec:nonzeroq}, we provide the classical quick proof over $\R$ and $\C$ using matrices. In Section \ref{sec:divisionalgebra}, we work over an arbitrary real division algebra, where matrices become less useful because the determinant is not multiplicative \cite{FAN2003193} and right-sided identities cannot be obtained due to a lack of associativity; this forces us to work with the partial convergents $T^{-1}_{a_m}\cdots T^{-1}_{a_n}0$. In Section \ref{sec:Iwasawa}, we adapt this approach to the broader setting of Iwasawa CFs. Note that in the Iwasawa setting, both quaternionic and octonionic CFs can be treated using real matrices, and a similar CF over $\R^n$ can be introduced.

Our proof of Lemma \ref{thm:non-convergence} is given in Section \ref{sec:non-convergence} and is inspired by previously-known proofs in $\R$ and $\C$. We sketch these proofs to give an idea of the methods, requiring that $a_n$ belong to $\mathbb{Z}$ in the real case, and belong to $\mathbb{Z}[i]$ in the complex case. We write $x_n=T_{a_n}\cdots T_{a_1}x_0$.
\begin{proof}[Sketch of the proof of Lemma \ref{thm:non-convergence} for $\R$ and $\C$]
In the real case, suppose the absolute value of $x_n$ converges to $1$. Then for sufficiently large $n$, either $\norm{x_n-1}<1/2$ or $\norm{x_n-(-1)}<1/2$. Supposing, without loss of generality, that such an $x_n$ is close to $1$, we must have that  $a_{n+1}=2$ and $x_{n+1}=\frac{1}{x_n}-2$ is now close to $-1=\frac{1}{1}-2$. Using \eqref{eq:inversionidentity}, we have:
\(\norm{x_{n+1}-(-1)} = \norm{\frac{1}{x_n}-\frac{1}{1}} = \norm{x_n}^{-1}\norm{1}^{-1}\norm{x_n-1}>\norm{x_n-1}.\)
Thus, once $x_n$ is sufficiently close to $\pm 1$, it gets pushed away from $\pm 1$ under iteration, contradicting the original assumption that it converges to $\pm 1$.

In the complex case, the proof is analogous: if $\norm{x_n}$ converges to $1$, then the requirement that  $x_n$ and $x_{n+1}=\frac{1}{x_n}-a_{n+1}$ must both be close to the unit circle (see Figure \ref{fig:ornament}) implies the stronger statement that $x_n$ must converge to the set $S_1$ (that is, for $\epsilon>0$ and sufficiently large $n$, each $x_n$ is $\epsilon$-close to some point of $S_1$). In particular, $S_1$ consists of those points $z\in\mathbb{C}$ for which there is an $a\in\mathbb{Z}[i]$ with $|z|=|z^{-1}-a|=1$. Again, using \eqref{eq:inversionidentity}, we can show that once $x_n$ is sufficiently close to $S_1$, it must be pushed away from $S_1$ under iteration, contradicting our assumption. 
\end{proof}
 
The fact that the set $\{\pm 1\}$ in the real case and the set $S_1$ in the complex case are both discrete sets is crucial to the proof; however, in higher dimensions, the natural choice of $S_1$ is no longer discrete. Thus, to prove Lemma \ref{thm:non-convergence} in full generality, we use a descent argument and build a nested sequence of sets $\Sph = S_0\supset S_1 \supset \ldots \supset S_n = \emptyset$, where each $S_i$, intuitively, consists of points $z\in \Sph$ where there are at least $i$ digits $a$ with $|z|=|z^{-1}-a|=1$. We then prove that convergence to each $S_i$ in fact implies convergence to $S_{i+1}$, eventually leading us to a contradiction.
 

\subsection{General Setting}
\label{sec:generalsetting}
We now recall the framework of Iwasawa CFs and state our Main Theorem \ref{thm:IwasawaMain}. We work in an Iwasawa Inversion space $\X_k^n$, built as follows (see \cite{lukyanenko_vandehey_2022}, which was motivated in part by the work of Chousionis-Tyson-Urba\'nski on Iwasawa groups \cite{MR4126256} and by the study of complex hyperbolic space \cite{MR1695450}).

Note that in the case $k=\R$, $\X^n_k$ reduces down to a parabolic model of $\R^n$ with the usual group law and metric. When $k=\C$, we obtain the Heisenberg group.

We start with an \emph{associative} real division algebra $k$ and work with vectors of the form $(a,b,c)\in k\times k^n\times k$ that are null vectors with respect to the Minkowski inner product $\langle {(a,b,c)}, {(d,e,f)}\rangle_{(n+1,1)}=-\overline c d+\overline b \cdot e-\overline a f$. The Minkowski norm of a vector $(a,b,c)$ is given by $\Norm{(a,b,c)}_{(n+1,1)}=-\overline a c+\norm{b}^2-\overline c a$, so null vectors satisfy the condition $\norm{b}^2-2\Re(\overline a c)=0$.

In its projective representation, the space $\X^n_k$ consists of null vectors, up to projective equivalence, with the exclusion of a point \emph{at infinity}. More precisely, the null vectors come in two forms: for the \emph{finite points} we have $a\neq 0$, so we may divide through by $a$ on the right to obtain $(1,u,v)=(1,ba^{-1},ca^{-1})$, which we write as $(u,v)$; if $a=0$, then we also have $b=0$ and may normalize $(a,b,c)$ as $(0,0,1)$, which we refer to as $\infty$. We will therefore think of $\X^n_k$ as the ``paraboloid''
\(\X^n_k=\{(u,v)\in k^n\times k \mst \norm{u}^2-2 \Re(v)=0\}.\)

We may furthermore identify a point $(u,v)$ with the matrix $\threevector{1&0&0}{u&\id&0}{v&\overline u&1}$, which sends null vectors to null vectors and in particular sends the point $(\vec 0,0)$ to the point $(u,v)$. This induces the group law $(u,v)*(u',v')=(u+u', v+\overline u u'+v')$.
The identity is $0=(\vec 0,0)$, and group inverses are given by $(u,v)\mapsto(-u,\overline v)$.

We equip $\X^n_k$ with the gauge $\Norm{(u,v)}=\sqrt{\norm{v}}$ and the Koranyi (also called Cygan or gauge) metric, characterized by  \(d((u,v),(u',v'))=d(0,(-u, \overline v)*(u',v')) \hspace{.15in} \text{and } \hspace{.15in} d(0,(u,v))=\Norm{(u,v)}=\sqrt{\norm{v}}.\)

The \emph{Koranyi inversion} $\iota_0$ is given by the matrix $\threevector{0& 0&-1}{0&\id&0}{-1&0&0}$, which sends $(a,b,c)\mapsto(-c,b,-a)$ and $(u,v)\mapsto(-uv^{-1}, v^{-1})$.
One shows (see Lemma \ref{lemma:inverseFormulaSiegel}) that the Koranyi inversion satisfies\footnote{In particular, the restriction of $\iota_0$ to the unit sphere $\Sph$ is an order-two isometry. Note, however, that when $k\neq \R$ this mapping is \emph{not} the restriction of a global isometry.} two critical properties of the Euclidean inversion:
\(\Norm{\iota_0 x}=\norm{x}^{-1} \hspace{1in} d(\iota_0 x, \iota_0 y)\Norm{x}\Norm{y}=d(x,y).\)

More generally, by an \emph{inversion} $\iota$ we will mean a composition $A\iota_0$ where $A$ is an isometry preserving the origin. This allows us, in particular, to represent the complex inversion $z\mapsto 1/z=\overline z/\norm{z}^2$ using real coordinates in $\X_\mathbb{R}^2$ as  $(x,y)\mapsto (x,-y)/(x^2+y^2) = \overline{\iota_0(x,y)}$.

We will think of an \emph{(Iwasawa) fraction} as a null vector $(q,r,p)$, associated to the point $(u,v)=(rq^{-1}, pq^{-1})$ that then satisfies $2 \Re v = \norm{u}^2$. Note that $q$ and $p$ are one-dimensional over $k$, while $r$ can be a vector. 

\begin{defi}An \emph{(Iwasawa) continued fraction} is an expression of the form $\iota \gamma_1 \iota \gamma_2 \cdots \iota \gamma_n(0)$ or a limit as $n\rightarrow \infty$ where each mapping $\gamma_i$ is an isometry of $\X_n^k$. 
\end{defi}

\begin{remark}
Note that in the context of Iwasawa CFs, our digits are isometric mappings $\gamma$, rather than numbers. This simplifies notation and gives a little more flexibility on the choice of digits (e.g., $\gamma$ might include a rotating or reflecting after translating, although we will not use this flexibility in this paper). Since adding a number is an isometry, this convention is compatible with the more familiar notion of digit.
\end{remark}

We will require two technical conditions on our continued fractions. In both cases, we suppose that the digits $\gamma_i$ lie in a discrete subgroup $\Zee\subset \Isom(\X_n^k)$.  Both of these are satisfied by most, but not all, Iwasawa CFs of interest.
\begin{defi}[Integrality Condition]
Let $\Mod=\langle \Zee, \iota\rangle$ be the \emph{modular group} generated by the possible digits $\Zee$ and the inversion $\iota$. We will say that the continued fraction is \emph{integral} if there is a discrete ring $R$ (necessarily, with unity) such that, for any matrix $M\in \Mod$ we have that the vector $(a,b,c)=M(1,0,0)$ satisfies $a,c\in R$.
\end{defi}

Note that in the above definition we don't expect that $b\in R^n$. Indeed, we often have $b\in \sqrt{2}R^n$.

\begin{defi}[Sphere-normalizing Condition]
\label{defi:inversioncommutes}
Given $\iota$ and $\Zee$, we say that the corresponding CF has the \emph{sphere-normalizing condition} if for every $\gamma\in \Zee$ there exists $\gamma'\in \Zee$ such that for all $s\in\Sph$:
\begin{enumerate}
    \item $\iota \gamma \iota s \in \Sph$ if and only if $\gamma' s\in\Sph$,
    \item for such points, $\iota \gamma \iota s = \gamma' s$.
\end{enumerate}
\end{defi}

We can now state our main result in full generality (cf.~Theorem \ref{thm:stronglyirrational} for alternate assumptions for norm-transcendental points):
\begin{thm}
\label{thm:IwasawaMain}
Consider an integral Iwasawa CF with space $X=\X^n_k$,  inversion $\iota$, and lattice $\Zee \subset \Isom(X)$ and $K$ a subset of the open unit ball $B(0,1)$, and $\overline K$ the closure of $K$.

Suppose one of the following is true:
\begin{enumerate}
    \item  The set of points $s\in \overline K \cap \Sph$ such that, for some $\gamma\in \Zee$, $\gamma \iota s \in \overline K \cap \Sph$, is finite,
    \item $X=\X^n_\R$ and the sphere-normalizing condition holds.
\end{enumerate}
Then, given a sequence $a_n\in \Zee$ and $x_0\in K$ such that ${T_{a_n}\cdots T_{a_1}x_0}\in K$ for all $n\geq 0$, the associated continued fraction converges to $x_0$:
\(x_0=\lim_{n\rightarrow \infty} T^{-1}_{a_1}\cdots T^{-1}_{a_n}0.
\)
\end{thm}

Theorem \ref{thm:IwasawaMain} follows immediately from Lemmas \ref{lemma:prodIntro}, \ref{thm:non-convergence}, and \ref{lemma:nonzero q}. We prove Lemmas \ref{lemma:prodIntro} and \ref{lemma:nonzero q} in the Iwasawa context in Section \ref{sec:Iwasawa}.  The proof of Lemma \ref{thm:non-convergence} depends on the choice of assumptions in Theorem \ref{thm:IwasawaMain}.

Let us sketch the proof of Lemma \ref{thm:non-convergence} under the first assumption, which uses the method of Dani-Nogueira \cite{DN2014} discussed in Section \ref{sec:outline} above, with minor adaptations. Suppose that the iterates $x_i$ converge to the boundary of the ball.  They must then remain close to points $s\in \Sph$ such that for some $\gamma \in \Zee$ we have that $\gamma \iota s\in \overline K \cap \Sph$: otherwise, they would be mapped to the interior of the unit ball. Thus, $x_i$ must be converge to the discrete set $\overline K \cap \Sph$. One shows that it therefore converges to one of the points in the discrete set, which then repels it, a contradiction.

In Section \ref{sec:non-convergence}, we elaborate on the above components, and also prove Lemma \ref{thm:non-convergence} under the second assumption of Theorem \ref{thm:IwasawaMain}, where we must take $k=\R$. 

\subsection{Examples}
\label{sec:examples}

For the purpose of examples, we will restrict ourselves to discussing CF \emph{algorithms}, where in addition to an underlying space $X$, an inversion $\iota$, and a lattice $\Zee\subset \Isom(X)$, we will have a fundamental domain $K\subset B(0,1)$ such that for almost all $x\in K$, there exists a unique $a$ such that $T_a x\in K$. We will then only consider a CF expansion such that $T_{a_n}\cdots T_{a_1}x_0\in K$ for all $n$. For such algorithms, almost every $x$ is associated to a unique sequence $a_1,a_2,\dots$ of digits.

Also, for our examples, we will restrict ourselves to considering those cases where $\Zee$ consists of translations and ignore the possibility of reflections or rotations. As such, we will often write that $\Zee$ equals a given set of numbers, such as $\Zee=\mathbb{Z}$ to indicate that $\Zee$ consists of the translations $x\mapsto x+a$ for $a\in \mathbb{Z}$.

Let us suppose for the moment that $X=\X_\R^n=\R^n$, which includes the possibility of $X$ being a real division algebra. When does the sphere-normalizing condition hold?

In the case where $\iota\vert_\Sph=\id\vert_\Sph$, we trivially get the sphere-normalizing condition with $\gamma'=\gamma$. In the case of a real division algebra, this corresponds to the inversion $\iota(x) = 1/\overline{x}$.

More generally, suppose that there is an order-two isometry $f: X\rightarrow X$ such that on $\Sph$ we have $\iota\vert_\Sph=f\vert_\Sph$. Suppose $\Zee$ is normalized by $f$: that is for each $\gamma\in \Zee$ we have $\gamma'\in \Zee$ such that $f\gamma=\gamma'f$. Then the sphere-normalizing condition is satisfied with this choice of $\gamma'$: for $x\in \Sph$, one has $\iota x = fx$, and thus $\gamma'x = \gamma'f\iota x = f \gamma \iota x$. If $\gamma'x\in\Sph$, then $f \gamma \iota x$ is further equal to $\iota \gamma \iota x$, providing inclusion in $\Sph$ and equivalence of the equations. Conversely, if $\iota \gamma \iota x \in \Sph$, then $\iota \gamma \iota x= f\gamma f x = \gamma'x$, also providing inclusion in $\Sph$ and equivalence of the equations. Thus an inversion like $\iota(x)=1/x$ (which acts on $\Sph$ by conjugation) will satisfy the sphere-normalizing condition provided $\Zee$ is closed under conjugation.

\begin{example}
The regular CF expansion that began this paper consists of $X=\mathbb{R}$, $\iota(x)=1/x$, $\Zee=\Z$, and $K=[0,1)$. A simple variant, called the $\alpha$-CF \cite{ST}, replaces the interval $K=[0,1)$ with the interval $K=[\alpha-1,\alpha)$ for a chosen $\alpha\in [0,1)$.\footnote{Many real CF variants use $\iota(x)=|1/x|$; however, these do not nicely extend to higher dimensions, so we will not consider them here.}
\end{example}

\begin{example}
The well-known A.~Hurwitz complex CFs are built from $X=\mathbb{C}$, $\iota(x)=1/x$, $\Zee=\Z[\ii]$, and $K=\{a+b\ii:a,b\in[-1/2,1/2)\}$. This $K$ is the Dirichlet region around the origin for $\Zee$ with a particular choice of boundary. See Figure \ref{fig:hurwitzandchevron}.
\end{example}

\begin{example}
The J.~Hurwitz complex CFs (which were independently rediscovered by Tanaka a century later) are built from $X=\mathbb{C}$, $\iota(x)=1/x$, $\Zee=\Z[\ii]\cdot(1+\ii)$, and 
\(
K= \left\{a(1+\ii)+b(1-\ii):a,b\in\left[-\frac{1}{2},\frac{1}{2}\right)\right\}.
\)
This is an improper CF algorithm and $\overline{K}\cap \Sph$ consists of the four points $\pm 1, \pm \ii$. 
\end{example}

\begin{example}
Starting with the J.~Hurwitz complex CFs as a template, we may construct an example where the sphere-normalizing condition fails. Suppose we replace the inversion with $\iota(x) = e^{\ii\pi/10}/\overline{x}$. Consider $s=e^{-\ii\pi/10}$ and $\gamma:x\mapsto x-2$. Then $\iota \gamma\iota s = e^{11\ii \pi/10}$. However, $\iota\gamma\iota s - s$ does not belong to $\Zee$, and thus we cannot find a $\gamma'$ with the desired property. However, as noted previously, $\overline{K}\cap \Sph$ is finite in this case, so we can use the other condition in the statement of Theorem \ref{thm:IwasawaMain}.
\end{example}

\begin{example}
In analogy to how one could create $\alpha$-CF expansions from regular CF expansions by moving the interval $[0,1)$, we could create new variants of the A.~Hurwitz complex CF by moving the square $\{a+b\ii:a,b\in[-1/2,1/2)\}$. We could shift this to the right as far as possible (while still maintaining the square shape) to obtain a region 
\(
K=\left\{a+b\ii: a\in\left[ \frac{\sqrt{3}}{2}-1,\frac{\sqrt{3}}{2}\right), b\in \left[-\frac{1}{2},\frac{1}{2}\right)\right\}.
\)
Here, we have an improper CF algorithm, but one where $\overline{K}\cap\Sph$ is finite, so that condition (1) will be trivially satisfied even if $\iota$ is changed so that the sphere-normalizing condition no longer holds.

If we are not worried about maintaining the square shape, then we could push this $K$ farther to the right to obtain the chevron-shaped region
\(
K = \left\{ a+b\ii: b\in \left[ -\frac{1}{2},\frac{1}{2}\right), \norm{a+b\ii}<1, \norm{a+b\ii-1}\ge 1  \right\}.
\)
See Figure \ref{fig:hurwitzandchevron}. Here we no longer have that $\overline{K}\cap\Sph$ is finite, so if the sphere-normalizing condition does not hold, then more careful analysis is required.
\end{example}

\begin{figure}
    \centering
     \hfill{}
     \begin{subfigure}[b]{0.45\textwidth}
         \centering
        \includegraphics[width=\textwidth]{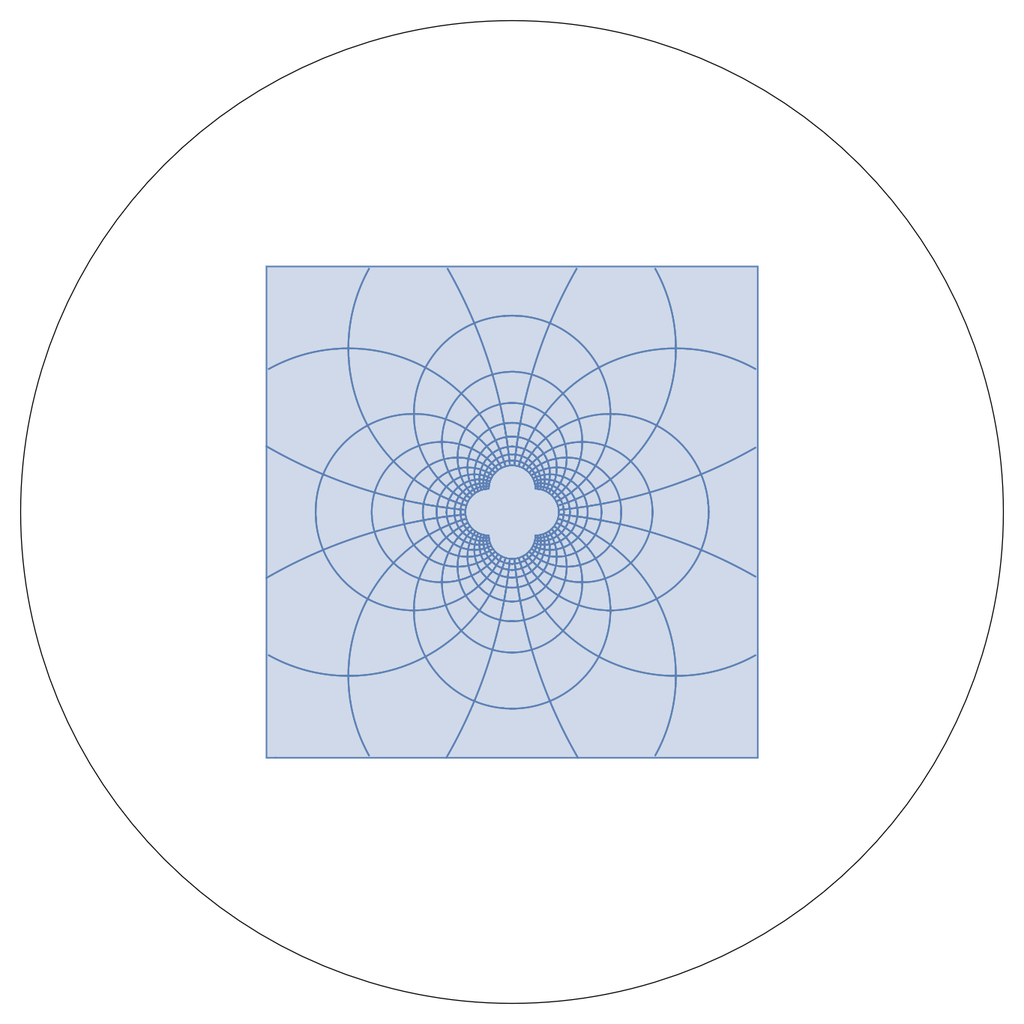}
         \caption{A.~Hurwitz CF}
     \end{subfigure}
     \hfill{}
     \begin{subfigure}[b]{0.45\textwidth}
         \centering
        \includegraphics[width=\textwidth]{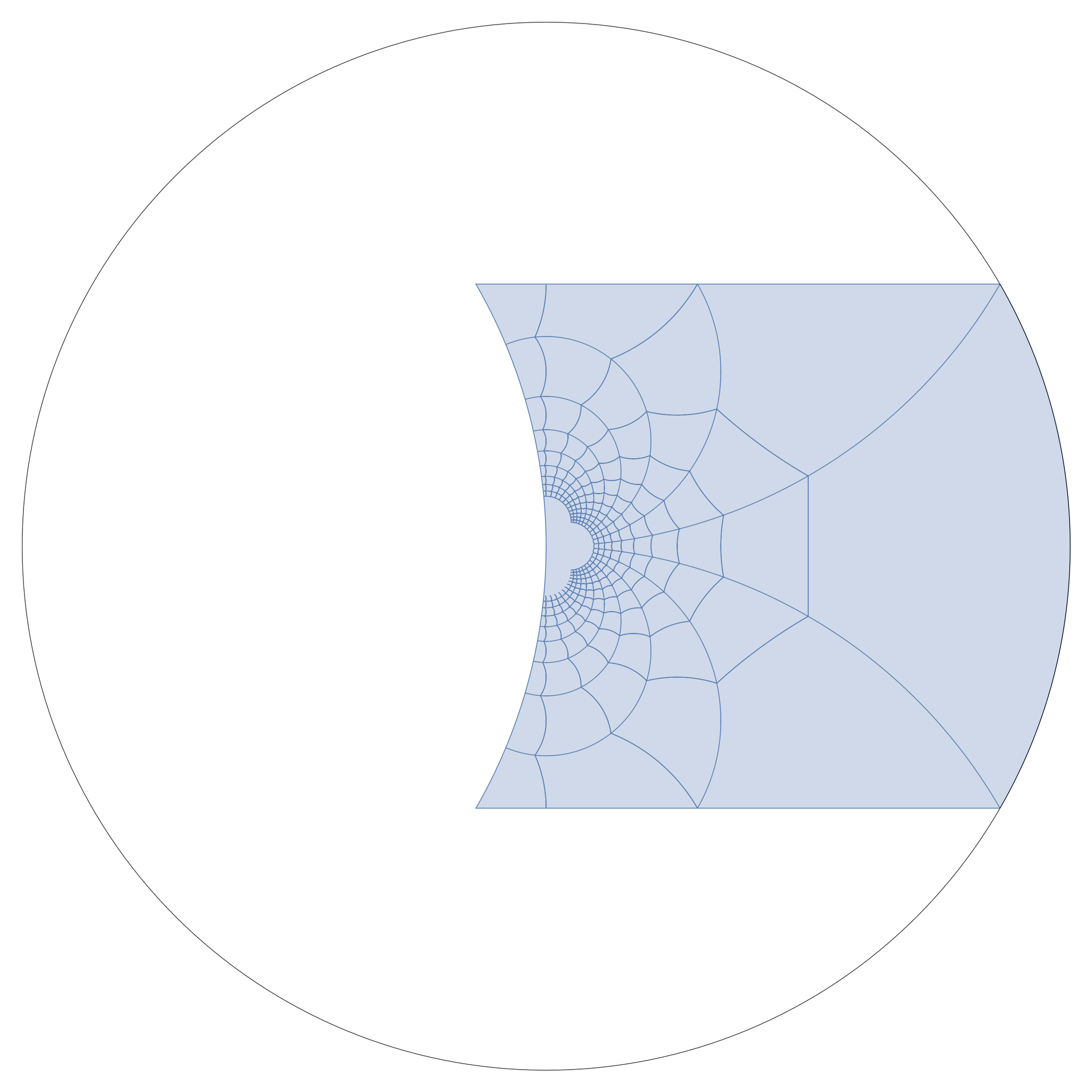}
         \caption{Chevron-type CF}
     \end{subfigure}
     \hfill{}
    \caption{Two complex CF algorithms. In each illustration, the domain $K$ is grey and is tiled by regions where $a_1$ is constant.}
    \label{fig:hurwitzandchevron}
\end{figure}

The above examples showcase results that are known either from classical results or from the aforementioned work of Dani and Nogueira. Subsequent examples, when $K$ is taken to be improper, are new.

\begin{example}
We can consider a three-dimensional continued fraction on $X=\X_\R^3= \R^3$ by considering $\iota(x_1,x_2,x_3)= (x_1,x_2,x_3)/\norm{(x_1,x_2,x_3)}^2$, $\Zee=\mathbb{Z}^3$, and $K=[-1/2,1/2)^3$. Convergence of the CF algorithm for this choice of $K$ was proven by the authors \cite{lukyanenko_vandehey_2022} using the properness condition. 

If a different, chevron-type choice of $K$ is made, then $\overline K$ can intersect the sphere in infinitely many points. This choice of $\iota$ satisfies the sphere-normalizing condition, providing convergence using condition (2) in Theorem \ref{thm:IwasawaMain}. Alternately, if we ignore that this $\iota$ is sphere-normalizing, we can recover convergence if it so happens that $\overline K$ contains only finitely many points $s\in \Sph$ where $\gamma\iota s\in \Sph$ for some $\gamma\in \Zee$ (in particular, all the points on the black lines in Figure \ref{fig:ornament}).
\end{example}

\begin{figure}
     \hfill{}
     \begin{subfigure}[b]{0.45\textwidth}
         \centering
        \includegraphics[width=\textwidth]{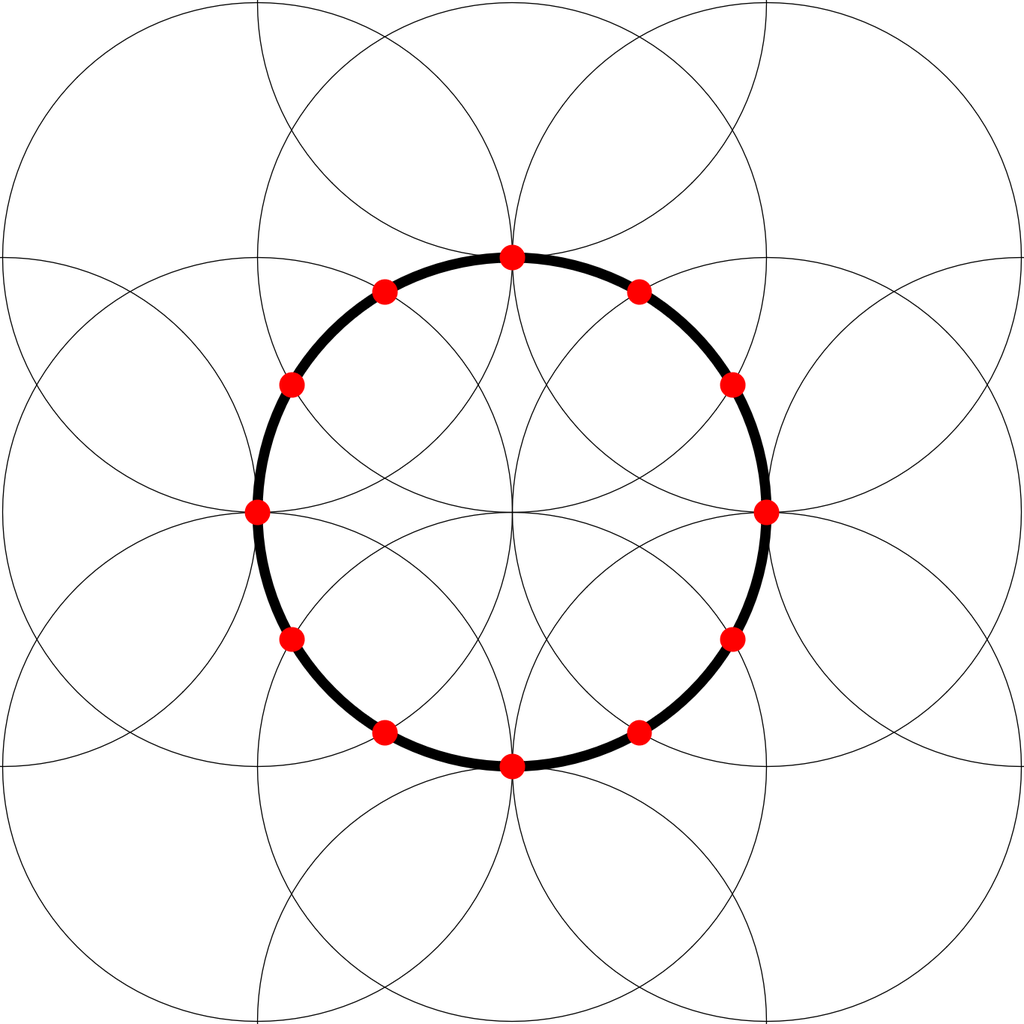}
         \caption{For $\R^2$ with digits in $\Z^2$, the sphere         $\Sph=S_0$, its translates (circles) and resulting intersections $S_1$ (dots).}
     \end{subfigure}
     \hfill{}
     \begin{subfigure}[b]{0.45\textwidth}
         \centering
        \includegraphics[width=\textwidth]{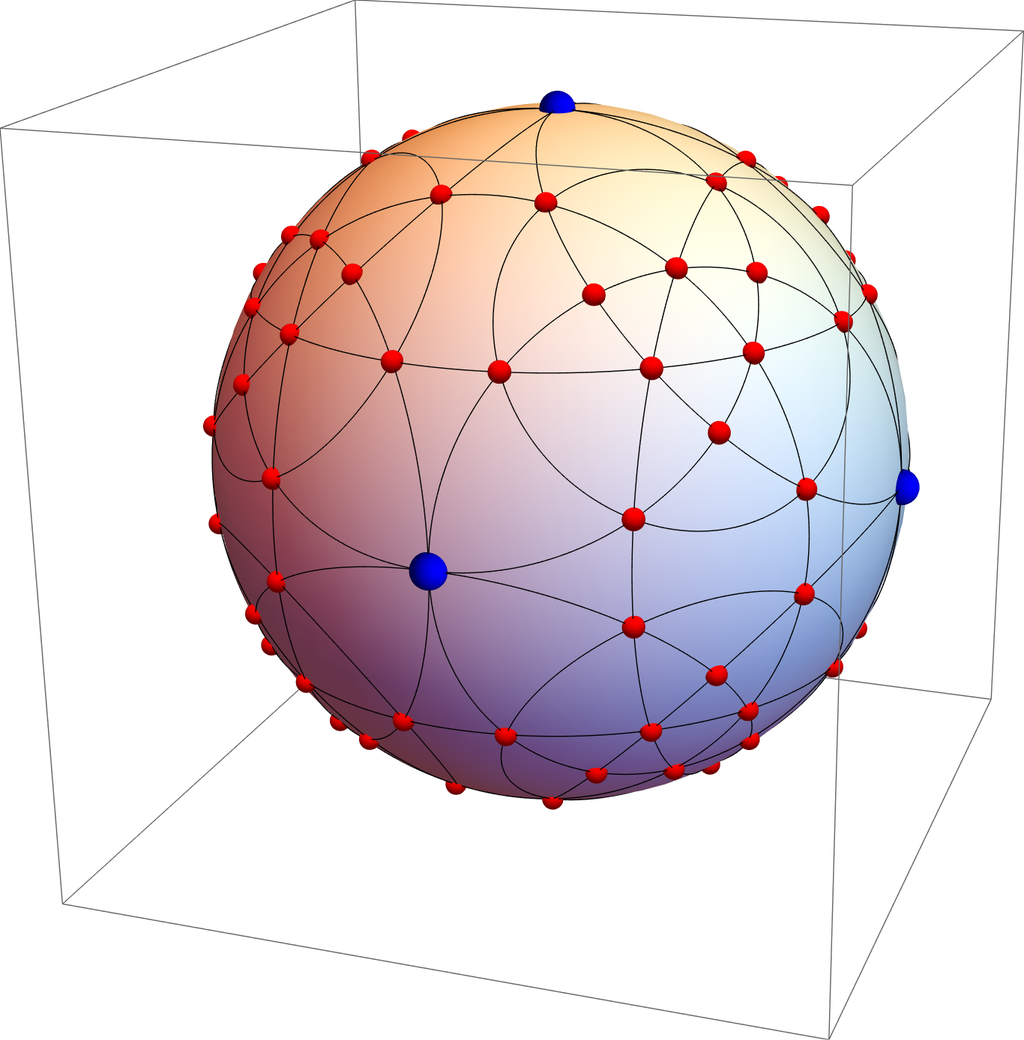}
         \caption{For $\R^3$ with digits in $\Z^3$, the sphere $\Sph=S_0$ and the sets  $S_1$ (circles) and $S_2$ (dots). The blue dots are pieces of $S^1$.}
     \end{subfigure}
     \hfill{}
    \caption{The boundary decomposition of the unit sphere.}
    \label{fig:ornament}
\end{figure}

\begin{example}
Analogously to how we could build a complex CF algorithm with $\Zee$ being any norm-Euclidean ring of integers, we can build a quaternionic CF algorithm on $\X^4_\R=\mathbb H$ with $\Zee$ any norm-Euclidean order\footnote{Recall that an order is a sub-ring that is also a lattice, and an order is norm-Euclidean if the Dirichlet region around the origin is strictly contained inside the unit ball.}. There are exactly three such orders  \cite{Fitzgerald-other}:
\begin{align*}
    & \Z\oplus \Z \ii \oplus \Z \jj \oplus \Z \left( \frac{1+\ii+\jj+\kk}{2}\right) \text{ (The Hurwitz integers }\mathcal{H}\text{)}\\
    & \Z\oplus \Z \ii \oplus \Z \left( \frac{1+\sqrt{3}\jj}{2}\right) \oplus \Z\left( \frac{\ii+\sqrt{3}\kk}{2}\right) \text{ (The Gausenstein integers)}\\
    & \Z \oplus \Z\left( \frac{2+\sqrt{2}\ii-\sqrt{10}\kk}{4}\right) \oplus \Z\left( \frac{2+3\sqrt{2}\ii+\sqrt{10}\kk}{4}\right) \oplus \Z\left( \frac{1+\sqrt{2}\ii+\sqrt{5}\jj}{2}\right).
\end{align*}
Again, we would use $\iota(x)=1/x$, with $\Zee$ as one of the above options, and $K$ as the Dirichlet domain. These would all result in proper CF algorithms. 

There are at least two natural choices of $\Zee$ that yield an improper CF algorithm. The first is the Lipschitz integers $\Z\oplus \Z \ii \oplus \Z \jj \oplus \Z \kk$, whose Dirichlet domain has radius $1$. The second is the an adjusted version of the Hurwitz integers $\mathcal{H}$ given by $\mathcal{H}\cdot(1+\ii)$; recall that the J.~Hurwitz CFs similarly replaced $\mathbb{Z}[\ii]$ with $\mathbb{Z}[\ii]\cdot(1+\ii)$. Note that the radius of the Dirichlet domain for the Hurwitz integers is $2^{-1/2}$ \cite[Thm.~2]{Mahler}, which implies that the radius of the Dirichlet domain for $\mathcal{H}\cdot(1+\ii)$ is $1$.
\end{example}

\begin{example}
For the octonions $\mathbb{O}$, we have two natural choices of lattice $\Zee$. The first are the Cayley integers\footnote{There are several isomorphic copies of the Cayley integers and any could be used.}
    \begin{align*}
        \mathcal{C} &= \mathbb{Z}\oplus \mathbb{Z}e_1\oplus \mathbb{Z}e_2\oplus \mathbb{Z}e_3\\
        &\qquad \oplus \mathbb{Z}h\oplus \mathbb{Z}e_1 h \oplus \mathbb{Z}e_2 h\oplus \mathbb{Z}e_3 h 
    \end{align*}
    with $h=(e_1+e_2+e_3-e_4)/2$. The Dirichlet domain for the Cayley integers has radius $2^{-1/2}$ \cite{Rehm}, so that these would produce a proper CF algorithm. 
    
    Alternately, one could use the lattice $\mathcal{C}\cdot (1+e_1)$, which would produce an improper CF algorithm.
    
    It appears to be an open question whether other norm-Euclidean orders exist over the octonions.
\end{example}

\begin{example}
We can also build an improper CF algorithm on $\X_{\mathbb{H}}^1$ using the Hurwitz integers $\mathcal{H}$ on $\mathbb{H}$ defined above. In particular, let $K_1$ be the Dirichlet domain for $\mathcal{H}$ on $\mathbb{H}$. Let $\mathcal{H}' = \mathbb{Z}i \oplus \mathbb{Z}j \oplus \mathbb{Z}k$, and let $K_2$ be the Dirichlet domain for $\mathcal{H}'$ in $\Im(\mathbb{H})$. In particular, $K_2$ will look like $[-1/2,1/2)^3$. Then using $\iota=\iota_0$ and $\Zee=\mathcal{H}\times \mathcal{H}'$ with $K=K_1\times K_2$ will give us a CF algorithm with radius of $K$ equal to $1$. See the discussion surrounding Example 3.14 in \cite{lukyanenko_vandehey_2022} for more information.
\end{example}


\subsection{Acknowledgements}We thank Florin Boca for his comments on the manuscript.

\section{Proof of Lemma \ref{lemma:nonzero q}}
\label{sec:nonzeroq}
We now prove Lemma \ref{lemma:nonzero q}, which interprets the product $\prod_{i=0}^{n-1}\norm{T^{-1}_{a_{i+1}}\cdots T^{-1}_{a_{n}}0}$ as $\norm{1/q_n}$ for a denominator $q_n$. We start with a closely-related lemma over $\R$ and $\C$ in Section \ref{sec:commutativeLemma}, then prove Lemma \ref{lemma:nonzero q} in Section \ref{sec:divisionalgebra} over any real division algebra, and adjust the proof for Iwasawa Inversion Spaces in Section \ref{sec:Iwasawa}.
\subsection{Proof over commutative real division algebras}
\label{sec:commutativeLemma}
Over $\R$ and $\C$, the proof of Theorem \ref{thm:DAmain} is simplified: we may replace Lemmas \ref{lemma:prodIntro} and \ref{lemma:nonzero q} with the following result:
\begin{lemma}
\label{lemma:commutativelemma}
For each $n$, there exist $p_n, q_n\in \Zee$, with $\norm{q_n}\geq 1$, such that \(p_n/q_n=T^{-1}_{a_1}\cdots T^{-1}_{a_n}0\) and such that
the distance from the CF approximant $T^{-1}_{a_1}\cdots T^{-1}_{a_n}0$ to $x_0$ satisfies
\[
\norm{T^{-1}_{a_1}\cdots T^{-1}_{a_n}0-x_0}= \frac{1}{\norm{q_n}}\prod_{i=0}^n \norm{T_{a_i}\cdots T_{a_1}x_0}. \label{eq:commutativelemma equation}
\]
\end{lemma}

\begin{proof}
Working over $\R$ or $\C$, we use the correspondence between linear-fractional mappings and 2-by-2 matrices to construct the continued fraction of depth $n$ using a matrix $M_n$. Namely, inversion is encoded by the matrix $\inversiontwomatrix$ and translation by $a_i$ is encoded by  $\translationtwomatrix{a_i}$. Let 
\(M_n =\inversiontwomatrix \translationtwomatrix{a_1}\cdots \inversiontwomatrix \translationtwomatrix{a_n}\)
and

Then $M_n$ encodes $T^{-n}=T_{a_1}^{-1}\cdots T^{-1}_{a_n}$, in the sense that for any $x\in k$ and $\twovector a b = M_n \twovector x 1$ we have that $\frac a b=T^{-n} x$. Likewise, $M_n^{-1}$ encodes $T^n=T_{a_n}\cdots T_{a_1}$. In particular, if we set $\twovector{p_n}{q_n}=M_n \twovector 0 1$, then the continued fraction of depth $n$ is then given by $T^{-n}0=p_n/q_n$. We assume for the moment that $q_n\neq 0$ and focus on the convergence of $p_n/q_n$ to $x$.

Now, an induction argument gives  $M_n=\twomatrix{p_{n-1}}{p_n}{q_{n-1}}{q_n}$ and, since $\det M_n=(-1)^n$, $M_n^{-1}=(-1)^n \twomatrix{q_n}{-p_n}{-q_{n-1}}{p_{n-1}}$. This allows us to calculate the distance between $p_n/q_n$ and $x$, since $\norm{\frac{p_n}{q_n}-x} = \frac{\norm{p_n-q_n x}}{\norm{q_n}}$ happens to be related to $T^nx$.
Namely, from $M_n^{-1}\twovector{x}{1} = (-1)^n \twovector{p_n-q_nx}{p_{n-1}-q_{n-1}x}$ we obtain $T^nx = \frac{p_n-q_nx}{p_{n-1}-q_{n-1}x}$. Multiplying together similar expressions for $T^nx, T^{n-1}x,\ldots, T^1x$, we obtain  $\prod_{i=1}^n T^ix = p_n-q_n x$, which gives the  identity
\(\frac{p_n}{q_n}-x = \frac{\prod_{i=1}^n T^i x}{q_n}\)
which immediately implies \ref{eq:commutativelemma equation}.
 
By discreteness of the ring $\Zee$, we have that $\norm{q_n}\geq 1$ if it is non-zero.

Set $x_n=T^nx$. Then, from the formula $x_{n-1}=T^{-1}_{a_n}x_n =  \cfrac{1}{x_{n}+a_n}$, we obtain
\(
    \inversiontwomatrix \translationtwomatrix{a_n} \twovector{x_n}{1} = \twovector {1}{x_n+a_n} = \frac{1}{x_{n-1}} \twovector{x_{n-1}}{1},
\)
which, by induction, implies
\begin{equation}
\label{eq:qnonzerocommutative}
M_n \twovector{x_n}{1} = \twovector{p_{n-1}x_n+p_n}{q_{n-1}x_n+q_n} = \frac{1}{x_0 \cdots x_{n-1}} \twovector{x_0}{1}
\end{equation}
 
Suppose that, in fact, $q_n=0$ but $q_{n-1}\neq 0$.
On the one hand, from \eqref{eq:qnonzerocommutative} we obtain $q_{n-1}=(x_0\cdots x_n)^{-1}$. Since we have $\norm{x_i} < 1$, this gives $\norm{q_{n-1}}>1$. On the other hand, the identity $\det(M_n)=(-1)^n$ reduces down to $p_n q_{n-1}=(-1)^n$ and therefore $\norm{p_n}=\norm{q_{n-1}}=1$, a contradiction. Thus $q_n$ is always non-zero and this completes the proof.
\end{proof}

\subsection{Proof over an arbitrary real division algebra}
\label{sec:divisionalgebra}


The proof of Lemma \ref{lemma:commutativelemma} does not apply over non-commutative real division algebras, since we no longer have a multiplicative determinant \cite{FAN2003193}. We therefore prove the more subtle Lemma \ref{lemma:nonzero q}.

We are interested in approximating $x_0$ with the continued fractions $T_{a_1}^{-1} T_{a_2}^{-1} \cdots T_{a_n}^{-1} 0$, which can be interpreted as fractions $P_nQ_n^{-1}$, for $P_n, Q_n \in \Zee$, by defining (denoting an empty list of digits by $\wedge$):
\begin{align*}
    P[\wedge]&=0 \qquad Q[\wedge]=1\\
P[a_n]&=1 \qquad Q[a_n]=a_n\\
    P[a_i,a_{i+1},\dots, a_n] &= Q[a_{i+1},a_{i+2},\dots, a_n]\\
    Q[a_i,a_{i+1},\dots, a_n] &= a_i Q[a_{i+1},a_{i+2},\dots, a_n]+P[a_{i+1},a_{i+2},\dots, a_n].
\end{align*}
and taking $P_n=P[a_1,\dots, a_n], Q_n=Q[a_1,\dots, a_n]$.

\begin{lemma}
\label{lemma:PQisT}
One has:
\(
P[a_1,\dots, a_n]Q[a_1,\dots, a_n]^{-1} =  T_{a_1}^{-1} T_{a_2}^{-1} \dots T_{a_n}^{-1} 0. 
\)
\begin{proof}Assume the result holds for all strings of length up to length $n-1$. Then,
\begin{align*}
 &T_{a_1}^{-1} T_{a_2}^{-1} \dots T_{a_n}^{-1} 0 \\ &\qquad= T_{a_1}^{-1}   ( P[a_2,\dots, a_n]Q[a_2,\dots, a_n]^{-1} )\\
 &\qquad= (P[a_2,\dots, a_n]Q[a_2,\dots, a_n]^{-1} + a_1)^{-1}\\
 &\qquad = (Q[a_2,\dots, a_n] Q[a_2,\dots, a_n]^{-1})(P[a_2,\dots, a_n]Q[a_2,\dots, a_n]^{-1} + a_1)^{-1}\\
 &\qquad = Q[a_2,\dots, a_n]\left(\left( P[a_2,\dots, a_n]Q[a_2,\dots, a_n]^{-1} + a_1\right)  Q[a_2,\dots, a_n]\right)^{-1}\\
 &\qquad = Q[a_2,\dots, a_n]\left( P[a_2,\dots, a_n]+ a_1 Q[a_2,\dots, a_n]\right)^{-1}\\
 &\qquad = P[a_1,\dots,a_n] Q[a_1,\dots,a_n]^{-1}
\end{align*}
as desired. Here we again use the fact that octonions are alternative in the fourth equality.
\end{proof}
\end{lemma}

\begin{lemma} 
\label{lemma:productisQ}
One has:
\(
\prod_{i=0}^{n-1} \left| T^{-1}_{a_{i+1}}\cdots T_{a_n}^{-1}0\right| = \frac{1}{|Q_n|},
\)
provided $Q[a_{i},\ldots,a_n]\neq 0$ for all $i\le n$.
\begin{proof}We use the recursive identities above to compute: 
\begin{align*}
    \prod_{i=0}^{n-1} \left| T^{-1}_{a_{i+1}}\cdots T_{a_n}^{-1}0\right|
   &= \prod_{i=0}^{n-1} \left| P[a_{i+1},\dots,a_n]Q[a_{i+1},\dots,a_n]^{-1}\right|\\
   & = \prod_{i=0}^{n-1} \frac{\left| P[a_{i+1},\dots,a_n]\right|}{\left|Q[a_{i+1},\dots,a_n]\right|}= \prod_{i=0}^{n-1} \frac{ \left| Q[a_{i+2},\dots,a_n]\right|}{\left|Q[a_{i+1},\dots,a_n]\right|}\\
   &= \frac{\left|Q[\wedge]\right|}{\left|Q[a_1,\dots,a_n]\right|}= \frac{1}{|Q_n|}.\qedhere
\end{align*}
\end{proof}
\end{lemma}

Combining with Lemma \ref{lemma:prodIntro}, we get:
\begin{cor}
\label{lemma:pqproduct}
We have that
\(
\norm{P_nQ_n^{-1}-x_0} = \frac{\prod_{i=0}^n |x_i|}{|Q_n|},
\)
provided $Q[a_i, \ldots, a_n]\neq 0 $ for each $i\le n$.
\end{cor}

\begin{lemma}
\label{lemma:Qnonzero}
For any $i<n$, we have $Q[a_i, \ldots, a_n]\neq 0$.
\end{lemma}
\begin{proof} Consider the following assertion:

\textbf{Assertion($m$):} Let $\{b_1, \ldots, b_{m}\}$ be a digit sequence and $x\in K$. If for all $0\leq j\leq m$ one has $\norm{T_{b_j} T_{b_{j-1}} \cdots T_{b_1} x}<1$, then for all $0< i\leq m$ one has $Q[b_i, \ldots, b_m]\neq 0$.

The assertion is immediate when $m=0$. Working inductively, we assume it holds for $m\leq n-1$ and prove it for $m=n$.

Let $b_1, \ldots, b_n$ be a digit sequence and $x\in k$ a point such for all $0\leq j\leq m$ one has $\norm{T_{b_j} T_{b_{j-1}} \cdots T_{b_1} x}<1$. Suppose, by way of contradiction, that $Q[b_i, \ldots, b_n](x)=0$ for some $0<i\leq n$. 

It suffices to assume $i=1$. Indeed, if $i>1$, then by the inductive assumption applied to the sequence $b_{i}, \ldots, b_n$ and the point $T_{b_{i-1}}\cdots T_{b_1}x$, we  have that $Q[b_i, \ldots, b_n]\neq 0$.

Next, observe that we may apply Corollary \ref{lemma:pqproduct} to the sequence $b_2, \ldots, b_n$ and the point $T_{b_1}x$ to obtain
\[
\label{eq:firstqzero}
\norm{P[b_2,\dots, b_n]Q[b_2,\dots, b_n]^{-1}-T_{b_1}x} = \frac{ \prod_{i=1}^n |x_i|}{|Q[b_2,\dots, b_n]|}.
\]
By our assumption about $Q_n(x)=0$, we have
\begin{align*}
    0&= Q_n(x) = b_1 Q[b_2,\dots,b_n]+P[b_2,\dots, b_n]\\
    -P[b_2,\dots, b_n]Q[b_2,\dots, b_n]^{-1} &= b_1,
\end{align*}
which then gives:
\(
    \norm{P[b_2,\dots, b_n]Q[b_2,\dots, b_n]^{-1}-T_{b_1}x} = \norm{-b_1-T_{b_1}x } = \norm{b_1+T_{b_1}x} =\norm{x^{-1}} = \norm{x}^{-1}.
\)
Combining this with \eqref{eq:firstqzero}, we have 
\(
\prod_{i=0}^n |x_i| = |Q[b_2,\dots, b_n]|.
\)
The left-hand side is strictly less than $1$, while the right-hand side is at least $1$, since $Q[b_2,\dots,b_n]$ is a non-zero integer. Thus our assumption that $Q_n[b_1,\dots,b_n]=0$ is false.
\end{proof}

\subsection{Proof over Iwasawa spaces}
\label{sec:Iwasawa}

We now work in the setting of Iwasawa inversion spaces. We will use matrices, as in Section \ref{sec:commutativeLemma} to define the approximating fractions, but then return to the approach of \ref{sec:divisionalgebra} for the remainder of the proof.

We work with digit sequences of the form $a_i=(\alpha_i, \beta_i)$ satisfying $2\Re \beta_i = \norm{\alpha_i}^2$. The CF map associated to prepending the digit $(\alpha, \beta)$ is then given by the matrix
\(B_{(\alpha,\beta)} =  \threevector{0&0&-1}{0&\id&0}{-1&0&0}\threevector{1 &0 &0}{\alpha &\id &0}{\beta &\overline \alpha &1}=\threevector{-\beta &-\overline \alpha&-1}{\alpha &\id& 0}{-1 &0 &0}
\)
Given a sequence of digits, we obtain the numerators and denominator of the associated CF as follows, starting with $Q[\wedge]=1, R[\wedge]=0, P[\wedge]=0$, by multiplying by $B_{a_{i}}$ on the left:
\begin{align*}
    &Q[a_i,\ldots,a_n]=-\beta Q[a_{i+1},\ldots, a_n]-\overline \alpha R[a_{i+1},\ldots, a_n]-P[a_{i+1},\ldots, a_n]\\
    &R[a_{i},\ldots, a_n]=\alpha Q[a_{i+1},\ldots, a_n]+R[a_{i+1},\ldots, a_n]\\
    &P[a_{i},\ldots, a_n]=-Q[a_{i+1},\ldots, a_n]
\end{align*}

One can also write down similar right-sided relations (note in the Iwasawa framework we don't work over the octonions, so matrix multiplication is associative).

The proof of Lemma \ref{lemma:nonzero q} is immediate as soon as we establish analogs of  Lemma \ref{lemma:PQisT} (relationship between $P/Q$ and $T$), Lemma \ref{lemma:productisQ} (product expression for $Q_n$, giving the distance-product formula in Corollary \ref{lemma:pqproduct}), and Lemma \ref{lemma:Qnonzero} ($Q_n\neq 0$).

The analog of Lemma \ref{lemma:PQisT} is immediate since we are working with an associative system. It therefore suffices to adjust the proofs of  Lemma \ref{lemma:productisQ} and  \ref{lemma:Qnonzero} to the Iwasawa setting. 

For completeness, we first provide a quick proof of the inversion identity for the Koranyi inversion (see also the proof on p.19 of \cite{MR2312336}):
\begin{lemma}
\label{lemma:inverseFormulaSiegel}
The inversion $\iota$ satisfies $d(x,y)=d(\iota x, \iota y) \norm{x}\norm{y}$.
\begin{proof}
We will need the quaternionic identities $\overline{ab}=\overline b \overline a$ and $(\overline a)^{-1} = \overline{a ^{-1}}$.

Let $x=(u,v)$ and $y=(z,w)$. Then $\norm{x}=\norm{v}^{1/2}$ and $\norm{y}=\norm{w}^{1/2}$, and
\(d(x,y)=\norm{(-u,\overline v)*(z,w)}=\norm{(-u+z,\overline v+w-\overline u z)}=\norm{\overline v+w-\overline u z}^{1/2}\)

Inverting $x$ and $y$ and applying the group law from Section \ref{sec:generalsetting}, we have:
\begin{align*}
    d(\iota x, \iota y) &= d\left( \iota \twovector{u}{v}, \iota \twovector z w         \right)
                = d(\left(
                  \twovector{-uv^{-1}}{v^{-1}},
                  \twovector{-zw^{-1}}{w^{-1}}
                  \right)\\
                &= \norm{
                  \twovector{uv^{-1}}{\overline{v^{-1}}}*
                  \twovector{-zw^{-1}}{w^{-1}}
                  }
                = \norm{
                  \twovector{uv^{-1}+-zw^{-1}}{\overline{v^{-1}}+w^{-1}+\overline{uv^{-1}}(-zw^{-1})}
                  }\\
                &= \norm{\overline{v^{-1}}+w^{-1}+\overline{uv^{-1}}(-zw^{-1})}^{1/2} = \norm{\overline{v^{-1}}+w^{-1}+\overline{v^{-1}}\overline{u}(-zw^{-1})}^{1/2}.
\end{align*}
Returning to the claim, we compute:
\begin{align*}
d(\iota x, \iota y)^2\norm{x}^2\norm{y}^2
&=\norm{v}\norm{w}\norm{\overline{v^{-1}}+w^{-1}+\overline{v^{-1}}\overline{u}(-zw^{-1})}\\
&= \norm{ \overline v(\overline{v^{-1}}+w^{-1}+\overline{v^{-1}}\overline{u}(-zw^{-1})) w}\\
&=\norm{\overline v +w - \overline u z} = d(x,y).\qedhere
\end{align*}
\end{proof}
\end{lemma}

Lemma \ref{lemma:productisQ} and Corollary \ref{lemma:pqproduct} take the following form in the Iwasawa setting, with the same proof. Note that we are using the gauge metric given by $\norm{(u,v)}=\norm{v}^{1/2}$, which affects the exponent on $\norm{Q_n}$.
\begin{lemma}We have that
\(\left| T^{-1}_{a_{i+1}}\cdots T_{a_n}^{-1}0\right| = \frac{1}{\norm{Q_n}^{1/2}}\)
provided $Q[a_i, \ldots, a_n]\neq 0$ for each $i\leq n$.
\end{lemma}
\begin{cor}We have that
\(d((R_nQ_n^{-1}, P_nQ_n^{-1}), x_0) = \frac{\prod_{i=0}^n \norm{x_i}}{\norm{Q_n}^{1/2}}\)
provided $Q[a_i, \ldots, a_n]\neq 0$ for each $i\leq n$.
\end{cor}

Lastly, we prove an analogue of Lemma \ref{lemma:Qnonzero}:
\begin{lemma}For any $i<n$, we have $Q[a_i, \ldots, a_n]\neq 0$.
\begin{proof}
The proof is parallel to that of Lemma \ref{lemma:Qnonzero}, apart from the following calculation:

Recall that $(Q_n,R_n,P_n)$ is always a null vector for the Minkowski inner product, satisfying $\norm{R_n}^2-2\Re(\overline Q_n P_n)=0$. If we have $Q_n=0$, then we conclude that $R_n=0$ as well. It follows that $(Q_n,R_n,P_n)=(0,0,P_n)$, corresponding to the point at infinity. Thus, the inverse of $(Q_n,R_n,P_n)$ corresponds to the origin, and  therefore $( R[b_2, \ldots, b_n]Q[b_2, \ldots, b_n]^{-1}, P[b_2, \ldots, b_n]Q[b_2, \ldots, b_n]^{-1})=(-\alpha_1, \overline{\beta_1})$, where $(-\alpha_1, \overline{\beta_1})$ is the additive inverse of $b_1=(\alpha_1, \beta_1)$.

We then compute:
\begin{align*}
    d((R[b_2,\dots, b_n]&Q[b_2,\dots, b_n]^{-1},P[b_2,\dots, b_n]Q[b_2,\dots, b_n]^{-1}),T_{b_1}x) \\
    &=  d((-\alpha_1, \overline{\beta_1}), T_{b_1}x)
    = \norm{b_1*T_{b_1}x} = \norm{\iota x}= \norm{x}^{-1}.
\end{align*}
The rest of the proof is as in Lemma \ref{lemma:Qnonzero}. Note that by the integrality assumption of Theorem \ref{thm:IwasawaMain}, $Q[a_i, \ldots, a_n]$ is contained in a discrete ring $R$, so we again have that if $Q[a_i, \ldots, a_n]\neq 0$ then $\norm{Q[a_i, \ldots, a_n]}\geq 1$, since otherwise taking powers of $Q[a_i, \ldots, a_n]$ would give elements of $R$ that are arbitrarily close to 0.
\end{proof}
\end{lemma}

\section{Proof of Lemma \ref{thm:non-convergence}}
\label{sec:non-convergence}

We now prove Lemma \ref{thm:non-convergence} for $X=\R^n$, where CFs are defined either using a real division algebra (in dimensions 1,2,4,8) or using the Iwasawa CF formalism, see Sections \ref{sec:generalsetting} and \ref{sec:Iwasawa}. As before, we will work with digits $\gamma\in \Isom(X)$, generalizing the mappings $x\mapsto x+a$ for $a\in \R^n$.

Fix an Iwasawa inversion space $X=\R^n$, a discrete group $\Zee\subset \Isom(X)$, and an inversion $\iota$. Fix a sequence $\sigma$ in $\Zee$, and use it to define, for $i\geq 0$, the corresponding mappings $T^ix=\gamma_i^{-1} \iota \cdots \gamma_1^{-1} \iota x$ and $T^{-i}x=\iota \gamma_1 \cdots \iota \gamma_i x$. Note that $T^{i}$ and $T^{-i}$ are inverses of each other, but otherwise the exponents are not additive under composition. Recall that for Lemma \ref{thm:non-convergence} we want to show that $|T^i x|$ does not converge to $1$.

\subsection{Convergence, orbits, and traps} We start with a discussion of convergence.

Fix a subset $A$ of the unit sphere $\Sph$. Given $r>0$, denote by $N_r(A)$ the $r$-neighborhood of $A$, i.e., the set of points $x$ such that $d(x,a)<r$ for some $a\in A$.

\begin{defi}Assume that $x\in X$ satisfies, for all $i\geq 0$, $\norm{T^ix}<1$. We say:
\begin{enumerate}
    \item $x$ \emph{$\epsilon$-orbits $A$} if, for all $i\geq 0$, $T^ix\in N_\epsilon(A)$.
    \item $x$ \emph{eventually $\epsilon$-orbits $A$}, if for sufficiently large $i$, $T^ix\in N_\epsilon(A)$.
    \item $x$ \emph{converges to $A$} if, for each $\epsilon>0$, $x$ eventually $\epsilon$-orbits $A$.
\end{enumerate}
\end{defi}

\begin{defi}[Voronoi neighborhood, see Figure \ref{fig:voronoi}]
Given $a\in A$ and $r\geq 0$, define the \emph{Voronoi neighborhood} $V_r(a, A)$ to be the set of points $x\in X$ such that $d(a,x)<r$ and furthermore $a$ realizes the distance from $x$ to $A$. Equivalently, $V_r(a,A)=N_r(A)\cap V(a,A)=N_r(a)\cap V(a,A)$, where
$$V(a,A)=\{x \in X \st \text{ for all b }\in A \setminus \{a\} \text{ one has } d(x,a) \leq d(x,b)\}$$ 
is the closed Voronoi cell around $a$. .
\end{defi}

\begin{figure}
    \centering
    \includegraphics[width=.25\textwidth]{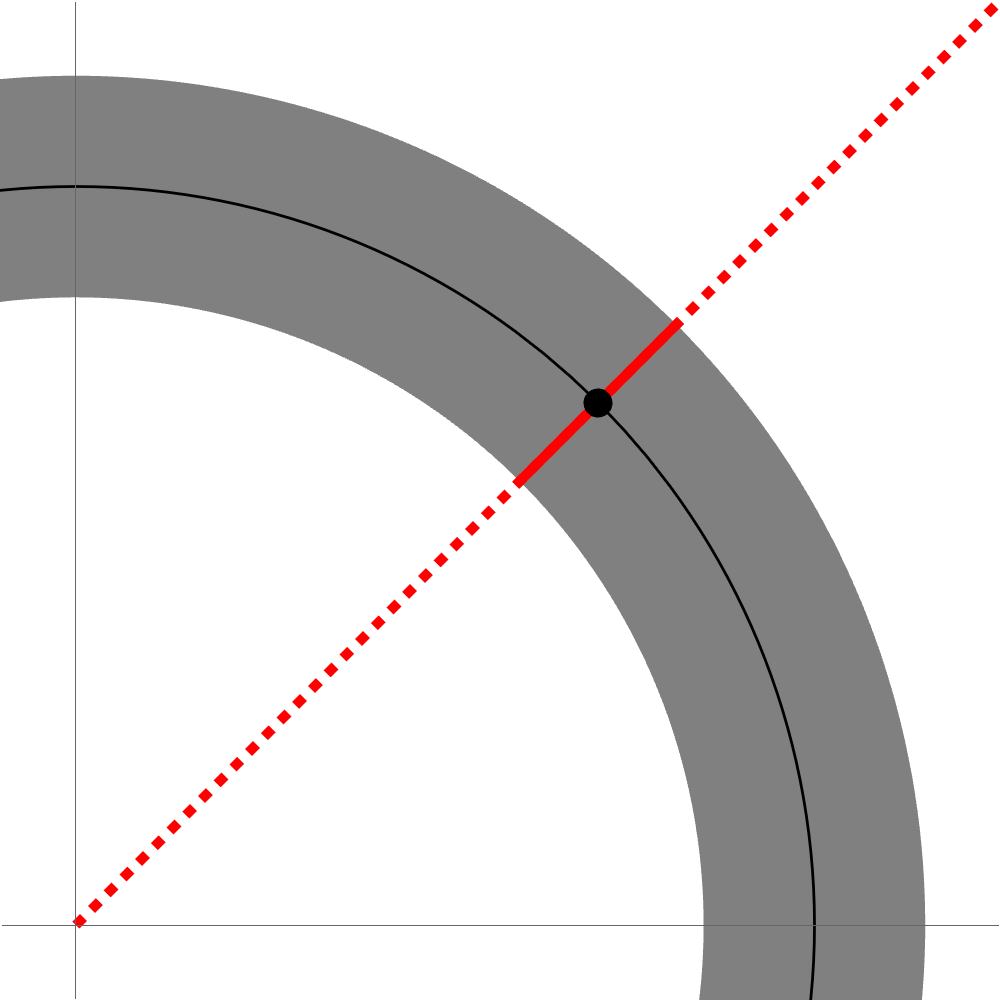}
    \caption{
    A depiction of a Voronoi neighborhood. The black line is $A$, the gray region is $N_r(A)$, the dotted red line is $V(a,A)$, and the solid red line is $V_r(a,A)$.}
    \label{fig:voronoi}
\end{figure}

\begin{defi}
Given $a\in A$ and $\epsilon>0$, we say that the point $a$ is an \emph{$\epsilon$-trap} (relative to $A$) if:
\begin{enumerate}
    \item for all $i>0$ one has $T^i a\in A$, and
    \item for any point $x$ which $\epsilon$-orbits $A$ and satisfies $x\in V_\epsilon(a, A)$, one has, for all $i\geq 0$, that $T^ix\in V_\epsilon(T^ia, A)$. 
\end{enumerate}  
We say that a point $x$ satisfying the conditions above is \emph{$\epsilon$-trapped} by $a$.
\end{defi}

\begin{lemma}
\label{lemma:trapOrbits}
If $x$ is $\epsilon$-trapped by a point $a\in A$, then its orbit $\{T^i x\}$ does not converge to $A$.
\begin{proof}
For all $i$, we have that the minimal distance from $T^i x$ to $A$ is given by $d(T^ix, T^ia)$. By the inversion formula and the fact that each $\gamma_i$ is an isometry, we have that
\(d(T^ix, A)=d(T^ix, T^ia) = \Pi_{j=0}^{i-1} \norm{T^j x}^{-1} d(x,a) > d(x,a)=d(x,A).\)
Note that, if the product converges quickly, it remains possible that $x$ is $\epsilon$-trapped.
\end{proof}
\end{lemma}

\subsection{Defining the sets $S_j$}

Starting with $S_0=\Sph$, we will now define a nested sequence of subsets $S_j$ of $\Sph$. We could quickly define $S_1$ with the desired properties as $\bigcup_{\gamma\in \Zee\setminus\{0\}} (S_0\cap \gamma S_0)$. However, we could not continue this definition to $S_2 = \bigcup_{\gamma\in \Zee\setminus\{0\}} (S_1\cap \gamma S_1)$, as this would just give us $S_1$ again. Instead we must analyze the individual pieces of $S_1$, namely the sets $S_0\cap \gamma S_0$, and see where translates of these intersect nontrivially with other pieces of $S_1$. Each of these intersections becomes one of the pieces of $S_2$. Nontrivial intersections of shifted pieces of $S_2$ become the pieces of $S_3$ and so on. The full definition is given in Definition \ref{defi:decomposition} below, and examples are given in Figure \ref{fig:ornament}.

The various pieces of $S_j$ will be indexed using superscripts: $S_j^1, S_j^2,S_j^3,\dots$.

\begin{remark}The sets $S_j$ appear to have the property that, for a point $s\in S_j$, we will have at least $j$ choices of a digit $a$ such that $T_a(s)\in S_j$. A variant of this property would then hold for points $x$ near $S_j$. However, we do not prove or use this property.
\end{remark}

\begin{defi}
\label{defi:decomposition}
We will define several sequences: 
\begin{itemize}
    \item $S_j$, each of which will be a subset of $\Sph$.
    \item $\operatorname{Pieces}_j$, each of which will be a set whose elements are subsets of $S_j$ and whose union is all of $S_j$. The elements of $\operatorname{Pieces}_j$ will be denoted by $S_j^1,S_j^2,S_j^3,\dots$.
    \item $K_j$, which are integers satisfying $K_j=|\operatorname{Pieces}_j|$.
    \item $\operatorname{Intersectors}_j$, each of which will be a set whose elements are subsets of $X$.
\end{itemize}

Let $S_0=\Sph$ and $\operatorname{Pieces}_0=\{\Sph\}$ and $K_0=|\operatorname{Pieces}_0|$. Then for each $j\geq 0$, we proceed to define the sequences iteratively.
\\For $1\leq k\leq K_j$,  the non-trivial intersect\underline{ors} of $S_j^k$: 
\(\operatorname{Intersectors}_{j+1}^k = \{\gamma S_j^{k'}: 1\leq k'\leq K_j, \gamma \in \Zee \text{, and }S_j^k \cap \gamma S_j^{k'} \notin \{\emptyset, S_j^k\} \},\)
For $1\leq k\leq K_j$,  the resulting non-trivial intersect\underline{ions} with $S_j^k$:
\(\operatorname{Pieces}_{j+1}^k = \{S_j^k \cap \gamma S_j^{k'}: \gamma S_j^{k'}\in \operatorname{Intersectors}_{j+1}^k\}
\)
All of the non-trivial intersectors:
\(\operatorname{Intersectors}_{j+1} =\bigcup_{k=1}^{K_j} \operatorname{Intersectors}_{j+1}^k,\)
All of the pieces (setting $K_{j+1}$ to be the number of elements in $\operatorname{Pieces}_{j+1}$):
\(\operatorname{Pieces}_{j+1} = \bigcup_{k=1}^{K_j}  \operatorname{Pieces}_{j+1}^k = \{S_{j+1}^{k'}\}_{k'=1}^{K_{j+1}},\)
And, finally, the set $S_{j+1}$:
\(S_{j+1}=\bigcup \operatorname{Pieces}_{j+1} = \bigcup_{k=1}^{K_{j+1}}S_{j+1}^{k}.\)
\end{defi}

We will require some straightforward lemmas:
\begin{lemma}
\label{lemma:spheres}
Let $A$ and $B$ be spheres of dimension $d_1$ and $d_2$, respectively, with $d_1\leq d_2$. Then $A\cap B$ is one of the following: the empty set, a single point, $A$, or a sphere of dimension at most $d_2-1$. Furthermore, there is a group of rotations that acts transitively on $A\cap B$ and leaves both $A$ and $B$ invariant.
\begin{proof}
Recall that a sphere of dimension $d$ is the intersection of the full sphere $S(x,r)$ with a hyperplane $H$ of dimension $d+1$ passing through $x$. Write $A=S(x_a, r_a)\cap H_a$ and $B=S(x_b, r_b)\cap H_b$. Then $A\cap B = S(x_a, r_a)\cap S(x_b, r_b) \cap H_a \cap H_b$. The result follows from this description and the standard fact that $S(x_a, r_a)\cap S(x_b, r_b)$ is a sphere of codimension 1.
\end{proof}
\end{lemma}

\begin{lemma}
\label{lemma:decompositionproperties}
For each $j$ we have:
\begin{enumerate}
    \item The set $S_j$ is a finite union of sets $\{S_j^k\}_{k=1}^{K_j}$, called \emph{pieces} of $S_j$,
    \item Each piece $S^k_j$ is isometric to either a sphere of dimension at most $n-j-1$ or is a single point,
    \item The set of pieces is $\iota$-symmetric: for each $k$ there exists $k'$ (possibly equal to $k$) such that $\iota S_j^k = S_j^{k'}$.
\end{enumerate}
Note that we do \emph{not} assume that the pieces are disjoint, and they may even be nested in each other or double-counted.
\begin{proof}For $j=0$, the statements are immediate, so let us assume $j>0$ and the lemma holds for $j-1$.

Finiteness of $\operatorname{Pieces}_j$ follows from the fact that each set $\operatorname{Pieces}_j^k$ is finite, which in turn follows from the fact that $\Zee$ is a discrete group of isometries and therefore acts properly.

The second claim follows immediately from Lemma \ref{lemma:spheres}.

The third claim uses the symmetry assumption for the lattice $\Zee$.
By construction, the given $S_j^k$ can be written as $\gamma S_{j-1}^{k_1}\cap S_{j-1}^{k_2}$, so that 
\(\iota S_j^k = \iota (\gamma S_{j-1}^{k_1}\cap S_{j-1}^{k_2}) = \iota \gamma S_{j-1}^{k_1}\cap \iota S_{j-1}^{k_2}.\)
Noting that we are working on the sphere and using the symmetry assumption, we may write $\iota \gamma = \gamma' \iota$ for some $\gamma'\in \Zee$, so that
$\iota S_j^k = \gamma'\iota  S_{j-1}^{k_1}\cap \iota S_{j-1}^{k_2}$. The result then follows from the inductive assumption that the sets $\iota  S_{j-1}^{k_1}$ and $\iota S_{j-1}^{k_2}$ are in $\operatorname{Pieces}_{j-1}$.
\end{proof}
\end{lemma}

\begin{lemma}[Proximity Lemma]
\label{lemma:epsilonrestriction}
There exists $\epsilon_0>0$ and a function $\tau(\epsilon)$ satisfying $\lim_{\epsilon\rightarrow 0}\tau(\epsilon)=0$ such that for any $\epsilon<\epsilon_0$ the following statements are true:\\
\begin{enumerate}
    \item Let $A$ and $B$ be two sets that appear as either intersectors or pieces in Definition \ref{defi:decomposition}. Then,
\(N_\epsilon(A)\cap N_\epsilon(B) \subset N_{\tau(\epsilon)}(A\cap B).\)
    \item If $j,j'\geq 0$ and $\gamma\in\Zee$ and $A$ has the form $\gamma S_j$ or $\iota \gamma S_j$, then
\(N_\epsilon(A)\cap N_\epsilon(S_{j'}) \subset N_{\tau(\epsilon)}(A\cap S_{j'}).\)
    \item Let $A$ be a piece and $d(x,A)<\epsilon$. Then $d(x,A)$ is realized by a unique point of $A$.
\end{enumerate}

\begin{proof}
Observe first that we are really interested in finitely many objects if $\epsilon<1/2$. For (1) and (3): by construction, there are finitely many pieces and intersectors. For (2):  for $\epsilon<1/2$ and sufficiently large $\gamma$ the intersection $N_\epsilon(\gamma S_j)\cap N_\epsilon (S_{j'})$ is empty because $\gamma S_j$ is close to infinity; and likewise $N_\epsilon(\iota \gamma S_j)\cap N_\epsilon (S_{j'})$ is empty because $\iota \gamma S_j$ is close to $0$. 

We may therefore provide a separate $\epsilon_0$ for each instance of the claims, and then take the minimum.

For (1 and 2), let $A$, $B$ each be a sphere or a point that appears in the claim. If $A \cap B$ is disjoint, we require $\epsilon_0$ sufficiently small that $N_{\epsilon_0}(A)\cap N_{\epsilon_0}(B)=\emptyset$. Otherwise, set $\tau(\epsilon)=\inf\{\delta: N_\epsilon(A)\cap N_\epsilon(B)\subset  N_\delta(A\cap B)\}$. It follows immediately that  $\tau(\epsilon)\geq \epsilon$ and that $\tau(\epsilon)$ is non-increasing. We have several cases for the upper bound: If either $A$ or $B$ is contained in the other, we have $\tau(\epsilon)=\epsilon$. If $A$ and $B$ are transversal, by linearizing at any intersection point (by Lemma \ref{lemma:spheres}, the choice of point does not matter) we obtain a constant $C$ such that $\tau(\epsilon)\leq C\epsilon$ for sufficiently small $\epsilon$. Lastly, if $A$ and $B$ are tangent and non-equal, then they must meet at a single point but with different curvature, so we obtain that $\lim_{\epsilon \rightarrow 0} \tau(\epsilon)=0$ without a linear bound.

For (3): if the piece is a sphere, then the claim is immediate if $\epsilon_0$ is smaller than the radius of the sphere; if the piece is a point, then the claim is trivial.
\end{proof}
\end{lemma}

\subsection{Autonomous Inhabited Traps}

\begin{defi}
A point $s\in S_j$ is:
\begin{enumerate}
    \item $\epsilon$-\emph{inhabited} if $V_\epsilon(s,S_j)$ contains a point that $\epsilon$-orbits $S_j$
    \item $\epsilon$-$\emph{autonomous}$ if it is $\epsilon$-inhabited and for all $i\geq 0$ one has $d(T^i s, S_{j+1})\geq \tau(3 \epsilon)$ for $\tau$ defined in Lemma \ref{lemma:epsilonrestriction}.
\end{enumerate}
\end{defi}
\begin{lemma}
\label{lemma:trap} Suppose $\epsilon<\min\{1/2, \epsilon_0/3\}$, for $\epsilon_0$ given by Lemma \ref{lemma:epsilonrestriction}.
If $s\in S_j$ is $\epsilon$-autonomous, then $s$ is an $\epsilon$-trap relative to $S_j$.
\begin{proof}
We need to show that, for all $i\geq 0$,  that $T^is\in S_j$ and that any point $x$ in $V_\epsilon(s, S_j)$ that $\epsilon$-orbits $S_j$ satisfies $T^ix\in V_\epsilon(T^is, S_j)$. It suffices to do this for $i=1$, i.e., to show that $Ts\in S_j$ and $Tx\in V_\epsilon(Ts,S_j)$.

Fix $k$ such that $s\in S_j^k$. Let $x$ be an element of $V_\epsilon(s,S_j)$ that $\epsilon$-orbits $S_j$. Then $d(Tx,S_j)<\epsilon$ and therefore for some $k'$ we have that $d(Tx, S_j^{k'})<\epsilon$, and in particular $Tx \in V_\epsilon(s', S_j)$ for some $s'\in S_j^{k'}$. We will show that $s'=Ts$.

We compute:
\begin{align*}
    d(T^{-1}s',x) &= d(\iota \gamma_1 s', \iota \gamma_1 Tx) = \frac{d(\gamma_1 s', \gamma_1 Tx)}{\norm{\gamma_1 s'}\norm{\gamma_1 Tx}} = \frac{d(s', Tx)}{\norm{\gamma_1 s'}\norm{\gamma_1 Tx}} \\
    &< \frac{\epsilon}{\norm{\gamma_1  s'}\norm{\gamma_1  Tx}}
    \leq \frac{\epsilon}{\norm{\gamma_1  s'}\norm{\iota x}}
    \leq \frac{\epsilon}{\norm{\gamma_1  s'}}\\
    &\leq \epsilon/(1-\epsilon) < 2\epsilon,
\end{align*}
where the last line follows from $d(s', Tx)<\epsilon$, which gives  $d(\gamma_1 s', \iota x)<\epsilon$ and thus
\(
\norm{\gamma_1 s'}\geq \norm{\iota x} - \epsilon \geq 1-\epsilon.
\)

We therefore conclude that $d(s, T^{-1}s')<3 \epsilon$.  We thus have that $d(s, T^{-1}S_j^{k'})<3\epsilon$.

Now, $T^{-1}S_j^{k'}=\gamma_1 \iota S_j^{k'}$. By the $\iota$-invariance of the pieces (part 3 of Lemma \ref{lemma:decompositionproperties}), there is a piece $S_j^{k_1}$ such that $S_j^{k_1}=\iota S_j^{k'}$.
We may therefore use the intersection part of Lemma \ref{lemma:epsilonrestriction} (recalling that $d(s,S_j^k)=0$) to conclude that $d(s, S_j^k\cap T^{-1}S_j^{k'})<\tau(3\epsilon)$. In particular, $S_j^k\cap T^{-1}S_j^{k'}$ is non-empty and we either have $S_j^k = T^{-1}S_j^{k'}$ or that $S_j^k\cap T^{-1}S_j^{k'}$ is a piece of $S_{j+1}$. The latter would imply that $d(s, S_{j+1})<\tau(3\epsilon)$, contradicting the assumption that $s$ is $\epsilon$-autonomous. 

We conclude that $T^{-1}S_j^{k'}=S_j^k$ and therefore $T^{-1}s'\in S_j^k\subset  \Sph$. Now that $\norm{T^{-1}s}=1$, we have that distances from $x$ to $s$ and $T^{-1}s'$ are distorted in the same way under $T$, so $d(s,x)<d(T^{-1}s',x)$ if and only if $d(Ts, Tx)<d(s',Tx)$. 

By the uniqueness of distance minimizers (Lemma \ref{lemma:epsilonrestriction}) in $\epsilon_0$-neighborhoods of $\operatorname{Pieces}_j$, we have $s'=Ts$, as desired. In particular:
\begin{enumerate}
    \item $Ts\in S_j$
    \item by $\epsilon$-orbiting, $x\in N_\epsilon(S_j)$
    \item since $s'=Ts$, $Tx\in V_\epsilon(Ts,S_j)$.\qedhere
\end{enumerate} 
\end{proof}
\end{lemma}

We now show that we found all of the points we are interested in. Note that most points in $\Sph$ will be far from $S_1$ but not inhabited.

\begin{lemma}
\label{lemma:inhabitedAutonomousTraps}
There is some function $\delta=\delta(\epsilon,\Zee,\iota)$ such that $\lim_{\epsilon\to 0} \delta =0$ and for $s\in S_j$ the following is true:

If $\epsilon>0$ is sufficiently small, $d(s, S_{j+1})>\delta$, and $s$ is $\epsilon$-inhabited, then $s$ is an $\epsilon$-trap.
\begin{proof}
Suppose $\epsilon$ and $s$ satisfy the assumptions in the lemma. 

We will prove $s$ is $\epsilon$-autonomous and conclude that it is an $\epsilon$-trap using Lemma \ref{lemma:trap}. We are already assuming that $s$ is $\epsilon$-inhabited, so it suffices to show that for all $i\geq 1$ we have that $d(T^is,S_{j+1})>\tau(3\epsilon)$. Let $\delta_0>\tau(3\epsilon)$ with an exact value to be determined later, and choose $\delta_i$ iteratively so that $\tau(2\delta_{i+1}) <\delta_i$ and $\delta_{i+1}<1/2$.

Assume that $d(s,S_{j+1})> \delta_0$. We now claim that $d(Ts, S_{j+1})>\delta_1$.  Suppose, by way of contradiction, that for some $s'\in S_{j+1}$ we have $d(Ts,s')\leq \delta_1$.

We first prove that $d(s, T^{-1}s')<2\delta_1$. We compute:
\(
d(s,T^{-1}s') =
d(s, \iota \gamma_1 s') =
\frac{d(\iota s, \gamma_1 s')}{\norm{\gamma_1 s'}} =
\frac{d(Ts, s')}{\norm{\gamma_1 s'}} \leq 
\frac{\delta_1}{\norm{\gamma_1 s'}} \leq
\frac{\delta_1}{1-\delta_1},
\)
where the last estimate is given by $\norm{\gamma_1 s'}\geq \norm{\iota s}-d(\iota s, \gamma_1 s') \geq 1- \delta_1$. Since we assumed that $\delta_1<1/2$, we conclude that $d(s, T^{-1}s')<2\delta_1$.

We thus have that $s\in S_j$ and $d(s, \iota \gamma_1 S_{j+1})<2\delta_1$. By the intersection part of Lemma \ref{lemma:epsilonrestriction}, this implies that $d(s, S_j \cap \iota \gamma_1 S_{j+1})< \tau(2\delta_1)<\delta_0$. Now, by the sphere-normalizing condition \ref{defi:inversioncommutes}, there is a $\gamma_1'\in \Zee$ such that 
\( S_j \cap \iota \gamma_1 S_{j+1} = S_j \cap \gamma_1'\iota  S_{j+1}.\)
Since $S_{j+1}$ is $\iota$-symmetric, this is furthermore equal to $S_j \cap \gamma_1' S_{j+1}$. It then follows directly from the definition of $S_{j+1}$ that $S_j \cap \gamma_1' S_{j+1} \subset S_j \cap S_{j+1} = S_{j+1}$, so our estimate becomes $d(s, S_{j+1})<\delta_0$, which contradicts our assumption.

As in the proof of Lemma \ref{lemma:trap}, we can use the fact that $s$ is inhabited and that $d(s,S_{j+1})>\tau(3\epsilon)$ to conclude that $Ts\in S_j$.  We thus have that $Ts\in S_j$ and $d(Ts, S_{j+1})>\delta_1$.

Now, we would like to apply the above argument recursively to the full orbit $\{T^is\}$ of $s$. Each time we have to apply it, the estimated distance to $S_{j+1}$ shrinks. In particular, performing $i$ applications would give that $d(T^i s, S_{j+1})>\delta_i$. Note that we require $\delta_i\ge \tau(3\epsilon)$ in order to apply the recursion a further step. 

On the other hand, we don't have to use the above argument if $T^is$ is a point that has already appeared in the sequence. Let $a$ denote the maximum possible number of distinct elements in the orbit $\{T^i s\}$ of $s$. We claim that we can take $a=2\norm{\{\gamma : \gamma \Sph \cap \Sph \neq \emptyset\}}$, which is finite since $\Zee$ is assumed to be discrete and proper. Indeed, by  the sphere-normalizing condition Definition \ref{defi:inversioncommutes}, as long as $T^i s$ stays in $\Sph$, it is always of the form $\gamma \iota s$ or $\gamma s$ for some $\gamma\in \Zee$, and therefore there are at most $a$ choices for the point $T^i s$, as desired.

So suppose we choose $\delta_0$ so that $\delta_a\ge \tau(3\epsilon)$ and define $\delta=\delta_0$. Then, provided $\epsilon$ is sufficiently small, we have that $d(T^i s, S_{j+1})\geq \tau(3\epsilon)$ for all $i$ as desired.
\end{proof}
\end{lemma}

We can now prove Lemma \ref{thm:non-convergence} as a corollary of Lemma \ref{lemma:inhabitedAutonomousTraps}.
\begin{cor}
[Non-convergence to the boundary]
Let $\norm{x}<1$ and let $\gamma_n$ be a sequence in $\Zee$ such that $\norm{T^n x}=\norm{\gamma_n^{-1}\iota \cdots \gamma_1^{-1} \iota x}<1$ for all $n$. Then it is not true that $\norm{T^n x}\rightarrow 1$.
\begin{proof}
Note first that if $\norm{T^n x}\rightarrow 1$, then $x$ $\epsilon$-orbits $S_0=\Sph$ for every $\epsilon>0$.

Working inductively, suppose that $x$ $\epsilon$-orbits some $S_j$. Let $s_i$ be the point that realizes $d(T^i x, S_j)$. If any $s_i$ satisfies $d(s_i,S_{j+1})>\delta$, then, by Lemma \ref{lemma:inhabitedAutonomousTraps}, we have that $T^i x$ is trapped by $s_i$ and hence, by Lemma \ref{lemma:trapOrbits}, the orbit $\{T^n x\}$ does not converge to $S_j\subset \Sph$. On the other hand, if $d(s_i,S_{j+1})\le \delta$ for all $i$, then $x$ must $\delta+\epsilon$-orbit $S_{j+1}$.

Thus, convergence to $\Sph$ implies convergence to $S_1$ and so on with dimension dropping by at least 1 each time, eventually culminating in convergence to an empty set, which is impossible.
\end{proof}
\end{cor}

\bibliographystyle{amsplain}
\bibliography{bibliography}

\end{document}